\documentclass[11pt,reqno]{amsart}
\usepackage{amssymb,mathtools,calc,verbatim,enumitem,tikz,url,hyperref,mathrsfs,cite,fullpage}
\usepackage{bbm}
\usepackage{stmaryrd}
\usepackage{textcomp}
\usepackage{setspace}
\usepackage{amssymb}
\usepackage{amsthm}
\usepackage{amsmath}
\usepackage{graphicx}
\usepackage{marvosym}
\usepackage{empheq}
\usepackage{latexsym}
\usepackage{fontenc}
\usepackage{color}
\usepackage{hyperref}
\usepackage{cleveref}
\usepackage{dsfont}

\newcommand{\ex}{\operatorname{ex}}

\usepackage{todonotes}

\newenvironment{poc}{\begin{proof}[Proof of claim]}{\end{proof}}

\newtheorem{theorem}{Theorem}[section]
\newtheorem{lemma}[theorem]{Lemma}
\newtheorem{corollary}[theorem]{Corollary}

\newtheorem{observation}[theorem]{Observation}
\newtheorem{proposition}[theorem]{Proposition}
\newtheorem{claim}[theorem]{Claim}
\newtheorem*{claim*}{Claim}
\newtheorem{problem}[theorem]{Problem}
	
\theoremstyle{definition}
\newtheorem{definition}[theorem]{Definition}

\newtheorem*{qu*}{Question}
\theoremstyle{remark}
\newtheorem{remark}[theorem]{Remark}

\newcommand{\Fq}{\mathbb F_q}

\newcommand\cC{\mathcal{C}}

\newcommand\cL{\mathcal{L}}

\renewcommand\le{\leqslant}
\renewcommand\ge{\geqslant}

	\def\<{\langle }
	\def\>{\rangle }

\begin{document}

\title{Constructor--Blocker games forbidding even cycles}
\author{
Lanchao Wang 
\and 
Zhifei Yan}

\address{School of Mathematics, Nanjing University, Nanjing, China, and ECOPRO, Institute for Basic Science, 55 Expo-ro, Yuseong-gu, Daejeon, 44126, Korea}\email{lanchaowang@foxmail.com}

\address{ECOPRO, Institute for Basic Science, 55 Expo-ro, Yuseong-gu, Daejeon, 34126, Korea}\email{zhifeiyan@ibs.re.kr}

\thanks{}

\begin{abstract}The Constructor--Blocker game is played on the edge set of $K_n$.
Two players alternately claim previously unclaimed edges. Constructor aims
to maximize the number of copies of a target graph $H$ in her graph while
keeping it $F$-free throughout the game, whereas Blocker aims to minimize
this number. When both players play optimally and Constructor moves first,
the final number of copies of $H$ in Constructor's graph is called the score
of the game and is denoted by $g(n,H,F)$. Recently, Balogh, Chen, and English
systematically studied this game for non-bipartite forbidden graphs $F$.
However, bipartite forbidden graphs present additional difficulties.

In this paper, we focus on the case in which the forbidden graph is an even
cycle. First, using a finite-field-geometric construction, we confirm a conjecture of Balogh,
Chen and English by proving 
\[
g(n,K_3,C_4)=\Theta(n^{3/2}).
\]
We also establish two-sided bounds for $g(n,K_3,C_{2k})$ that relate the
game score to classical extremal numbers. Our proofs reveal a structural
distinction between the cases $k=2$ and $k\ge3$: a vertex-duplication method
works for longer even cycles, but it necessarily creates $4$-cycles and
therefore cannot be used in the $C_4$-free game. Finally, we determine the exact
order of $g(n,C_t,C_{2k})$ for all $k\ge2$ and all $t\ge4$.
\end{abstract}
	
	\maketitle

\section{Introduction}

The study of Tur\'an-type problems is fundamental to extremal graph theory. One primary goal is to determine the generalized Tur\'an number $\ex(n, H, F)$, which represents the maximum number of copies of a graph $H$ within an $n$-vertex graph that avoids a forbidden subgraph $F$. Recently, Patk\'os, Stojakovi\'c, and Vizer \cite{PSV24} introduced a dynamic variant of this problem, known as the Constructor--Blocker game: Two players take turns claiming unclaimed edges on a complete graph $K_n$. Constructor's aim is to maximize her copies of $H$ without ever forming $F$, while Blocker plays without any restrictions to minimize this number. When both players play optimally and Constructor goes first, the final count of $H$ in Constructor's graph is called the score of the game, denoted by $g(n, H, F)$.
Since Constructor's final graph is always  $F$-free, the number of copies of $H$ she builds is at most the maximum possible number in a static graph. This gives the immediate upper bound 
$$g(n, H, F) \le \ex(n, H, F).$$

The game is well understood when the forbidden graph $F$ is not bipartite. For instance, Balogh, Chen, and English \cite{BCE25} proved a general theorem for any forbidden graph with  chromatic number  at least three.

\begin{theorem}[Balogh, Chen, and English]
Let $F$ be a graph with $\chi(F)>r\ge 2$. 
Then the game score is $g(n, K_r, F)=\Theta(n^r)$.
\end{theorem}

This result relies on a vertex-partition strategy. Constructor can  avoid a non-bipartite graph by dividing the vertices into several groups and only claiming edges between different groups. This allows her to safely build a dense graph. Using similar structural advantages, Balogh, Chen, and English \cite{BCE25} also solved the game when both the target and the forbidden graphs are odd cycles. However, this method fails when the forbidden graph is bipartite. Constructor can no longer use this partition trick, which forces her to play in a sparse environment. This fundamental change makes the game much harder, leaving large gaps between the known bounds. 

Because of this difficulty, recent progress on bipartite graphs has been strictly limited to acyclic structures. For instance, Boisson, Mogge, Parreau, and Pierron \cite{BMPP25} expanded the scope by forbidding short paths and stars, which are trees and structurally much simpler. They also studied a variation where Constructor must maintain a planar graph instead of avoiding a specific local subgraph. However, solving the game with arbitrary bipartite forbidden graphs, particularly even cycles, remains a challenging open problem.

\subsection{Counting cliques}

A natural target graph for Constructor is a clique $K_r$, particularly the triangle $K_3$. When exploring this target against a forbidden odd cycle, Balogh, Chen, and English \cite{BCE25} established the following bounds, where $\cC_{\le 2k} := \{C_4, C_6, \dots, C_{2k}\}$ denotes the family of even cycles up to length $2k$.

\begin{theorem}[Balogh, Chen and English]
Let $k \ge 2$. Then 
$$\Omega_k\bigl(\ex(n, \mathcal{C}_{\le 2k})\bigr) \le g(n, K_3, C_{2k+1}) \le O_k\bigl(\ex(n, C_{2k})\bigr).$$
\end{theorem}

In the same work, they investigated the game forbidding $C_4$, proving an upper bound of $O(n^{3/2})$ alongside a lower bound of $\Omega(n^{3/2} e^{-c\sqrt{\log n}})$. Believing the upper bound to be tight, they conjectured that $g(n, K_3, C_4) = \Theta(n^{3/2})$.

In this paper, we confirm this conjecture by developing an algebraic strategy to provide the missing lower bound. Furthermore, we investigate the game for all even cycles, leading to the discovery of a connection between the game score and classical extremal numbers. This general relationship is captured in our main result, which establishes a sandwich bound for the game score.

\begin{theorem} \label{thm:cliques}
Let $k \ge 2$. Then
$$\Omega_k\bigl(\ex(n,\cC_{\le 2k})\bigr) \le g(n,K_3,C_{2k}) \le O_k\bigl(\ex(n,C_{2k})\bigr).$$
\end{theorem}

By bridging the game setting with traditional extremal graph theory, this theorem immediately determines the exact order of the game score whenever the lower and upper bounds share the same polynomial growth rate. It is well known that, for \(k\in\{2,3,5\}\), the extremal numbers for
\(\mathcal C_{\le 2k}\) and for the single cycle \(C_{2k}\) have the same
order of magnitude, by results of Bondy--Simonovits, Erdős--Rényi--Sós,
Benson, and Wenger~\cite{BS,
ERS,B,W}. Hence we determine the cases $C_4$, $C_6$ and $C_{10}$.

\begin{corollary}
For $k\in\{2,3,5\}$, we have $g(n,K_3,C_{2k})=\Theta(n^{1+1/k}).$
\end{corollary}

The lower bounds in Theorem~\ref{thm:cliques} reveal a clear structural difference between the $C_4$-free game ($k=2$) and the games forbidding longer $C_{2k}$ ($k\ge 3$). Since Constructor must avoid the forbidden cycle, her strategy depends directly on its length. For $k \ge 3$, the cycle is long enough to allow shorter local cycles. We use this by starting with a high-girth bipartite graph and duplicating vertices on one side, which immediately forms many triangles. This duplication only introduces new cycles of length 3 or 4. These short cycles do not violate the $C_{2k}$ restriction, and the underlying high-girth prevents them from forming a larger $C_{2k}$.
However, this vertex duplication method naturally generates $C_4$ subgraphs, so it fails for the $C_4$-free game. To build triangles without creating $C_4$, the edges must be placed much more carefully. We solve this by applying finite field geometry to construct the required triangles. Because this case requires a  different mechanism, we analyze $g(n, K_3, C_4)$ separately from the cases where $k \ge 3$.

While we have a complete picture for these specific cases, the behaviour for larger cliques or longer cycles remains open. For larger cliques with $r\ge 2k$, the game score is trivially zero. Specifically, 
we also prove that $g(n,K_{2k-1},C_{2k})=0$; the proof is presented in the Appendix. However, for general cycles $C_{2k}$ where $k \ge 4$, we cannot determine the correct order for $g(n, K_r, C_{2k})$, even for $r=3$.

\subsection{Counting cycles}

A natural extension of the game is to consider the case where both the target graph and the forbidden graph are cycles. In classical extremal graph theory, the generalized Tur\'an number $\ex(n, C_t, C_k)$ is a heavily studied topic. For the dynamic game, Balogh, Chen, and English \cite{BCE25} resolved the case where both cycles are odd, proving that the game score matches the static extremal number.

\begin{theorem}[Balogh, Chen, and English]
Let $k > \ell \ge 2$. Then $g(n, C_{2\ell+1}, C_{2k+1}) = \Theta(n^\ell)$.
\end{theorem}

In this paper, we focus on forbidden even cycles. The extremal number $\ex(n, C_t, C_{2k})$ is fully understood. Gerbner, Gy\H{o}ri, Methuku, and Vizer \cite{GGMV20} analyzed the case where the target cycle is also even, while Gishboliner and Shapira \cite{GS20} completely settled the problem for all target cycles.  We prove that Constructor can always achieve the same order as $\ex(n, C_t, C_{2k})$  in the game, determining the score for all target cycles of length at least $4$.

\begin{theorem} \label{thm:cycles}
Let $k \ge 2$, and let $t \ge 4$ be an integer with $t \neq 2k$. Then the game score satisfies
\[
g(n,C_t,C_{2k})
=
\begin{cases}
\Theta_{t,k}\bigl(n^{t/2}\bigr), & \text{if } k=2,\\[2mm]
\Theta_{t,k}\bigl(n^{\lfloor t/2\rfloor}\bigr), & \text{if } k\ge 3.
\end{cases}
\]
\end{theorem}

Observe that if $t=2k$, then trivially $g(n,C_t,C_{2k})=0$.
The upper bound in this theorem follows immediately from the  results in \cite{GGMV20, GS20}. The primary contribution of our work lies in the lower bounds, which require two  different strategies depending on the length of the forbidden cycle. 

For longer even cycles where $k \ge 3$, we design a structural template graph. This template acts as a partial blow-up of the target cycle. We secure a small set of core vertices and build linear-size common neighbourhoods between consecutive core vertices. By carefully claiming edges in bulk, Constructor forces copies of $C_t$ by selecting one vertex from each neighbourhood. We prove that any subgraph of this template only contains cycles of length 4 or exactly $t$. Since we assume $t \neq 2k$ and $k \ge 3$, this guarantees that every move Constructor makes is safe from $C_{2k}$.

However, because this template method essentially relies on local $4$-cycles, it is not applicable when the forbidden graph is $C_4$. To resolve the exceptional $k=2$ case, we extend the finite field algebraic construction developed for triangles. By analyzing paths within a $C_4$-free incidence geometry and applying the pigeonhole principle, we show that Constructor can safely close enough paths to achieve the required bound for every target cycle covered by the theorem.

\subsection{Organization}

The remainder of this paper is organized as follows. In Section~\ref{sec:triangles}, we prove Theorem~\ref{thm:cliques} for counting triangles. We split the proof based on the forbidden cycle length. Subsection~\ref{subsec:triangles-k3} introduces the vertex duplication strategy for $k \ge 3$, and Subsection~\ref{subsec:triangles-k2} presents the algebraic incidence geometry approach for the $4$-cycle case. In Section~\ref{sec:cycles}, we apply these foundational concepts to prove Theorem~\ref{thm:cycles} for counting general cycles. 
This is handled by utilizing a structural template graph in Subsection~\ref{subsec:cycles-k3} and extending the finite field construction in Subsection~\ref{subsec:cycles-k2}. In Section~\ref{sec:concluding}, we provide concluding remarks and state an open problem. Finally, Appendix~\ref{app:additional} contains the proof that the game score for large cliques against even cycles is exactly zero.

\medskip

\section{Counting triangles} \label{sec:triangles}

In this section, we prove Theorem \ref{thm:cliques}. The upper bound is a direct consequence of the static extremal number. Since Constructor's final graph is always  $C_{2k}$-free, the number of triangles she can build is at most $\ex(n, K_3, C_{2k})$. By the result of Gishboliner and Shapira~\cite{GS20},
\[
\ex(n, K_3, C_{2k}) \le O_k(\ex(n, C_{2k})).
\]
Therefore, the upper bound naturally holds.

The main challenge lies in the lower bounds. We split the proof into two subsections based on the length of the forbidden cycle. For $k \ge 3$, we develop a vertex duplication strategy. This strategy safely generates many triangles by copying vertices in a bipartite graph. To solve the $k=2$ case, we provide a different algebraic construction in the second subsection.

\medskip

\subsection{The lower bound for longer even cycles ($k \ge 3$)} \label{subsec:triangles-k3}

The goal of this subsection is to establish the lower bound for all $k \ge 3$.

\begin{proposition} \label{prop:lower-k-3}
For any $k \ge 3$, the game score satisfies 
\[
g(n, K_3, C_{2k}) \ge \Omega_k(\ex(n, \cC_{\le 2k})).
\]
\end{proposition}

We approach this bound by using a vertex duplication strategy. We first construct a dense bipartite base graph that avoids short cycles. We then define a duplication operation and prove that creating clone vertices on one side of this graph forms triangles without introducing any new $C_{2k}$. Finally, we provide a three-phase strategy for Constructor to safely build this structure against Blocker. 

To make this approach precise, we begin by defining the duplication operation.

\begin{definition}[Duplication graph]
Let $B=(X,Y;E_B)$ be a bipartite graph. The duplication graph $\Gamma(B)$ is obtained from $B$ by adding a clone vertex $x'$ for every $x \in X$. The vertex set is $X \cup Y \cup X'$, where $X':=\{x':x \in X\}$. The edge set consists of the original edges $xy \in E_B$, the vertical edges $xx'$ for $x \in X$, and the clone edges $x'y$ whenever $xy \in E_B$.
\end{definition}

Before applying this operation generally, it is helpful to understand its behaviour when the base graph is acyclic.

\begin{lemma} \label{lem:forest-duplication}
If $T=(X,Y;E_T)$ is a forest, then every cycle in $\Gamma(T)$ has length 3 or 4.
\end{lemma}
\begin{proof}
For each $x \in X$, let $Q_x$ be the induced subgraph of $\Gamma(T)$ on the vertices $\{x, x'\} \cup N_T(x)$. By construction, each $Q_x$ is formed by taking a complete bipartite graph with parts $\{x, x'\}$ and $N_T(x)$, and then adding the edge $xx'$. Every cycle completely inside a single $Q_x$ clearly has length 3 or 4.

\begin{claim}
A cycle in $\Gamma(T)$ cannot contain edges from two different pieces.
\end{claim}
\begin{poc}
We define an auxiliary bipartite graph $T'$ with parts $\{Q_x\}_{x \in X}$ and $Y$. We add an edge between $Q_x$ and $y$ if $xy \in E(T)$. This graph $T'$ is isomorphic to $T$, so $T'$ is also a forest.

Assume for contradiction that a cycle $C$ in $\Gamma(T)$ uses edges from at least two different pieces. We take the smallest subtree of $T'$ that contains all pieces intersecting $C$. This subtree must have a leaf piece, say $Q_x$. Since $Q_x$ is a leaf in $T'$, it shares at most one vertex $y \in Y$ with all other pieces used by $C$. Any path of $C$ that enters $Q_x$ must leave $Q_x$ through the same vertex $y$. This forces $C$ to visit $y$ twice, which is impossible for a simple cycle.
\end{poc}

Therefore, every cycle of $\Gamma(T)$ is completely contained within a single $Q_x$, and hence has length 3 or 4.
\end{proof}

Building on this local property, we can now establish the global safety of the duplication graph when the base graph avoids cycles up to length $2k$.

\begin{lemma} \label{lem:duplication-safe}
Let $k \ge 3$. If $B=(X,Y;E_B)$ contains no cycle of length at most $2k$, then $\Gamma(B)$ contains no $C_{2k}$.
\end{lemma}
\begin{proof}
Assume for contradiction that $\Gamma(B)$ contains a cycle $C$ of length exactly $2k$. We define a subgraph $H \subseteq B$ by keeping the edge $xy \in E_B$ whenever $C$ uses either $xy$ or $x'y$. The vertical edges $xx'$ do not add edges to $H$. Since $C$ has $2k$ edges, $H$ has at most $2k$ edges.

If $H$ contains a cycle, the length of this cycle is at most $2k$. This contradicts our assumption that $B$ has no such cycles. Therefore, $H$ must be a forest. By our construction, $C$ is a subgraph of $\Gamma(H)$. However, since $H$ is a forest, Lemma \ref{lem:forest-duplication} implies that every cycle in $\Gamma(H)$ has length 3 or 4. This contradicts the fact that $C$ has length $2k \ge 6$. Thus, $\Gamma(B)$ is  $C_{2k}$-free.
\end{proof}

Finally, our strategy requires a dense base graph. We guarantee its existence using a standard scaling lemma, which follows immediately from
the result of Katona, Nemetz, and Simonovits~\cite{KNS} that, for any graph
family \(\mathcal F\), the ratio
\[
\frac{\operatorname{ex}(n,\mathcal F)}{\binom{n}{2}}
\]
is non-increasing in \(n\).

\begin{lemma} \label{lem:scaling}
Fix $0 < \alpha < 1$. Let $N = \lfloor \alpha n \rfloor$. For every graph family $\mathcal{F}$, we have 
\[
\ex(N,\mathcal{F}) \ge \frac{\alpha^2}{4}\ex(n,\mathcal{F}).
\]
\end{lemma}

With these structural and extremal tools established, we now proceed to the formal proof of Proposition \ref{prop:lower-k-3}. Constructor will execute her strategy in three distinct phases.

\begin{proof}[Proof of Proposition \ref{prop:lower-k-3}]
We set
\[
N = \lfloor n/20 \rfloor.
\]
We choose an $N$-vertex graph $G_0$ that contains no cycles of length at most $2k$, such that $e(G_0) = \ex(N, \cC_{\le 2k})$. We fix a bipartite subgraph $B = (X,Y; E_B)$ of $G_0$ such that $|X| \le |Y|$ and $e(B) \ge e(G_0)/2$. By Lemma \ref{lem:scaling}, we have 
\[
e(B) \ge c \cdot \ex(n, \cC_{\le 2k})
\]
for some constant $c > 0$. We call a vertex fresh if neither player has claimed an edge incident to it. 

\vspace{0.2cm}
\noindent\textbf{Phase I: Building disjoint stars}

For each $x \in X$, Constructor chooses a fresh vertex and labels it $x$. She then chooses two additional fresh vertices $a_x$ and $b_x$, and claims the two edges $xa_x$ and $xb_x$.

\begin{claim}
After Phase I, the board contains at least $|Y|$ fresh vertices.
\end{claim}
\begin{poc}
This phase requires $2|X|$ Constructor moves. Since Blocker makes an equal number of responses, the two players claim at most $4|X|$ edges in total. These edges touch at most $8|X| \le 4N \le n/5$ vertices. Since $|Y| \le N \le n/20$, the number of touched vertices is small enough to guarantee that Constructor can find $|Y|$ fresh vertices.
\end{poc}

Constructor assigns the labels of $Y$ to these fresh vertices. No edges are claimed in this labeling step. Because all vertices in $Y$ are currently fresh, all possible edges between $X$ and $Y$ are entirely unclaimed. With the vertices successfully prepared, we move to the next phase.

\vspace{0.2cm}
\noindent\textbf{Phase II: Claiming edges of $B$}

Constructor now plays exclusively on the unclaimed edges of $B$. She continues this until every edge of $B$ is claimed by one of the two players. Let $C_B \subseteq E_B$ be the set of edges claimed by Constructor.

\begin{claim}
Constructor successfully claims $|C_B| \ge e(B)/2$ edges, and every move in this phase is legal.
\end{claim}
\begin{poc}
Since she moves first and Blocker can take at most one edge of $B$ per round, she clearly claims at least half of the total edges in $B$. At any point in this phase, Constructor's graph is a subgraph of $B$ together with the leaf vertices $a_x$ and $b_x$. Because $B$ has no cycles of length at most $2k$, and adding leaves to a graph does not create new cycles, every edge of $B$ is legal to claim.
\end{poc}

As Constructor has secured a dense bipartite skeleton, she is ready to build the actual triangles.

\vspace{0.2cm}
\noindent\textbf{Phase III: Closing triangles}

For each $x \in X$, Constructor must choose one of the two leaves $a_x$ or $b_x$ to be active. We call this chosen leaf $\ell_x$. She makes this choice carefully to minimize the interference from Blocker. For a fixed $x$, she compares the number of Blocker-owned edges of the form $a_x y$ with the number of Blocker-owned edges of the form $b_x y$, where $y \in Y$ and $xy \in C_B$. She sets $\ell_x$ to be the leaf with the smaller number of such Blocker-owned edges.

\begin{claim} \label{claim:unclaimed-candidates}
By choosing the active leaves in this way, Constructor guarantees a set $Q$ of at least $e(B)/2$ unclaimed candidate edges of the form $\ell_x y$.
\end{claim}
\begin{poc}
Let $M = e(B)$. The total number of edges in $B$ claimed by Blocker is exactly $M - |C_B|$. The number of Blocker moves not spent on $B$ is at most $|C_B| - (M - |C_B|) = 2|C_B| - M$. Only these non-$B$ moves could have been spent on the candidate edges connecting the leaves to $Y$.

Because Constructor chose the active leaf to minimize Blocker's interference, the total number of Blocker owned candidate edges involving the active leaves is at most half of the available non $B$ Blocker moves. This is at most $(2|C_B| - M)/2$.

Therefore, the number of still unclaimed active candidate edges $\ell_x y$ where $xy \in C_B$ is at least
\[
|C_B| - \frac{2|C_B| - M}{2} = \frac{M}{2} = \frac{e(B)}{2}.
\]
We define $Q$ to be this set of unclaimed candidate edges, so $|Q| \ge e(B)/2$.
\end{poc}

In this final phase, Constructor only claims edges from $Q$. Whenever she successfully claims an edge $\ell_x y \in Q$, she forms the triangle $x \ell_x y$ because she already owns $x \ell_x$ and $xy$. By Lemma \ref{lem:duplication-safe}, claiming these edges is  legal because her final graph, after deleting vertices of degree one, is a subgraph of the duplication graph $\Gamma(B)$. Blocker can take at most one edge of $Q$ per round. Thus, Constructor claims at least $|Q|/2 \ge e(B)/4$ edges, which yields $\Omega_k(\ex(n, \cC_{\le 2k}))$ triangles.
\end{proof}

\begin{remark}
The vertex duplication strategy fails when the forbidden graph is a $4$-cycle. If a vertex $x$ is duplicated to a clone $x'$, and $x$ has at least two common neighbours $y_1$ and $y_2$ in the base graph $B$, the vertices $x, y_1, x', y_2$ immediately form a $C_4$. To prevent this, every vertex in $X$ could have at most one neighbour in $Y$, which means the base graph could only have $O(n)$ edges. This makes it impossible to build a dense graph. Therefore, we cannot use this structural method for the $C_4$-free game.
\end{remark}

\medskip

\subsection{The lower bound for the $4$-cycle ($k=2$)} \label{subsec:triangles-k2}

In this subsection, we prove the matching lower bound for the $4$-cycle.

\begin{proposition} \label{prop:lower-k-2}
The game score for the $4$-cycle satisfies 
\[
g(n, K_3, C_4) \ge \Omega(n^{3/2}).
\]
\end{proposition}

To overcome the limitations of the duplication method, we use finite field geometry. Motivated by a construction of Conlon \cite{Conlon21}, we define a bipartite incidence graph between points and lines. We then introduce an auxiliary graph based on an algebraic relation to add edges between points. We will prove that the union of these two graphs is $C_4$-free, and then use the pigeonhole principle to show Constructor can safely claim enough edges to form the required number of triangles. 

To put this plan into action, we first construct the base geometric structure. Let $q$ be an odd prime power. We partition the finite field $\Fq$ into two disjoint sets $S$ and $T$, such that $|S|$ is approximately $q/3$ and $|T|$ is approximately $2q/3$. We define a set of points $P = T \times \Fq$. We also define a set of lines $\cL$ containing all lines of the form $y = mx + b$, where the slope $m$ belongs to $S$ and the intercept $b$ belongs to $\Fq$. The total number of vertices is $|P| + |\cL| = q^2$. Let $B$ be the bipartite incidence graph between $P$ and $\cL$, where a point $(x,y)$ is connected to a line $L$ if the point lies on the line.

The most crucial property of this incidence graph is that it  avoids $4$-cycles.

\begin{lemma} \label{lem:incidence-C4-free}
The incidence graph $B$ is  $C_4$-free and has $e(B) = \Theta(q^3)$ edges.
\end{lemma}
\begin{proof}
Any two distinct points can lie on at most one line together. Since $B$ only contains non-vertical lines, a $C_4$ in $B$ would mean two points share two lines, which is geometrically impossible. Therefore, $B$ is $C_4$-free. The number of edges is exactly $|S||T|q$, which gives
\[
e(B) = \Theta(q^3),
\]
as required.
\end{proof}

To actually build triangles, Constructor needs to connect two points that share a common line. We define an auxiliary graph $F_\theta$ to provide these edges safely. For each $\theta \in \Fq$, we connect two distinct points $p=(x,y)$ and $p'=(x',y')$ in $F_\theta$ if $y+y' = xx' + \theta$.

\begin{lemma} \label{lem:polarity-C4-free}
For any $\theta \in \Fq$, the union graph $B \cup F_\theta$ is $C_4$-free.
\end{lemma}
\begin{proof}
First, we show $F_\theta$ itself is $C_4$-free. If a point $u=(u_x, u_y)$ is a common neighbour of $p=(x,y)$ and $p'=(x',y')$ in $F_\theta$, then we have the system of equations $y + u_y = x u_x + \theta$ and $y' + u_y = x' u_x + \theta$. Subtracting these equations gives $y - y' = u_x(x - x')$. If $x \neq x'$, then $u_x$ and subsequently $u_y$ are uniquely determined. If $x = x'$ and $p \neq p'$, the system has no solution. Thus, any pair of points has at most one common neighbour, making $F_\theta$ $C_4$-free.

We already know $B$ is $C_4$-free. A mixed $4$-cycle must consist of a line $L$ and a point $p_1$ in $B$, a point $p_0$ connected to $p_1$ in $F_\theta$, another point $p_2$ connected to $p_0$ in $F_\theta$, and $p_2$ must be connected to $L$ in $B$. Let $p_0 = (x_0, y_0)$. Any neighbour $(u_x, u_y)$ of $p_0$ in $F_\theta$ satisfies $u_y = x_0 u_x + \theta - y_0$. This means all $F_\theta$ neighbours of $p_0$ lie on a specific line with slope $x_0$. Since $p_1$ and $p_2$ are $F_\theta$ neighbours of $p_0$, they must lie on this line. However, $p_1$ and $p_2$ also lie on the line $L$ in $B$. This forces $L$ to have slope $x_0$. This is a contradiction, since \(x_0\in T\), whereas the slope of \(L\) must belong to \(S\), and \(S\cap T=\emptyset\).
Therefore, no mixed $C_4$ can exist.
\end{proof}

Having verified the safety of our algebraic construction, we are now fully equipped to prove Proposition \ref{prop:lower-k-2}.

\begin{proof}[Proof of Proposition \ref{prop:lower-k-2}]
We play the game on $q^2 \le n$ vertices, ignoring the rest. By standard number theory, we can find an odd prime $q$ between $\frac{1}{2}\sqrt{n}$ and $\sqrt{n}$. Constructor first claims only edges of $B$, stopping when every edge of $B$ is claimed by one of the two players. Let $G_B$ be the graph of edges claimed by Constructor. Because she moves first, she claims at least half of the edges, giving 
\[
e(G_B) = \Omega(q^3).
\]

\begin{claim}
Constructor's graph $G_B$ contains a set $\mathcal{P}$ of $\Omega(q^4)$ ordered paths of length two of the form $p L p'$, where $p$ and $p'$ are points and $L$ is a line.
\end{claim}
\begin{poc}
We can clean up $G_B$ by repeatedly removing vertices of very low degree. This leaves a dense subgraph $G_B'$ that still has $\Omega(q^3)$ edges and a minimum degree proportional to $q$. In $G_B'$, starting from any point $p$, we can find a line $L$ and then another point $p'$ on $L$. This forms a path of length two, corresponding to an open triangle. Since the minimum degree is high, $G_B'$ contains $\Omega(q^4)$ such ordered paths. Let $\mathcal{P}$ be the set of these paths.
\end{poc}

For any path $\Pi = p L p' \in \mathcal{P}$ with $p=(x,y)$ and $p'=(x',y')$, we define a specific value 
\[
\theta(\Pi) = y + y' - x x'.
\]
By definition, the edge between the points $p$ and $p'$ belongs exactly to the graph $F_{\theta(\Pi)}$. Adding this edge to $\Pi$ creates a triangle.

\begin{claim}
There exists a choice of $\theta \in \Fq$ such that the unclaimed edges of $F_\theta$ can close $\Omega(q^3)$ paths from $\mathcal{P}$.
\end{claim}
\begin{poc}
During the first phase, Blocker claimed at most $O(q^3)$ edges between points. A single edge between points can act as the closing edge for at most $O(1)$ paths in $\mathcal{P}$. This is because if $p$ and $p'$ are fixed, they can share at most one line in $B$. Thus, Blocker has ruined at most $O(q^3)$ paths.

The total number of available paths in $\mathcal{P}$ is still $\Omega(q^4)$. Since there are $q$ possible values for $\theta$, the pigeonhole principle implies there must exist a specific $\theta \in \Fq$ such that the unclaimed edges of $F_\theta$ can close at least $\Omega(q^3)$ paths.
\end{poc}

Constructor fixes this $\theta$. She assigns each available edge in $F_\theta$ a weight equal to the number of paths it can close. She then claims these candidate edges, always prioritizing the one with the maximum remaining weight. By Lemma \ref{lem:polarity-C4-free}, claiming these edges is  legal because the union graph $B \cup F_\theta$ is $C_4$-free.

Blocker can take at most one candidate edge per round. Since Constructor always takes the heaviest available edge, she secures at least half of the total weight. This means she completes $\Omega(q^3)$ triangles. Since $q$ is proportional to $\sqrt{n}$, the number of completed triangles is
\[
\Omega(q^3) = \Omega(n^{3/2}),
\]
as required.
\end{proof}

Since the general upper bound is established at the beginning of this section, Propositions \ref{prop:lower-k-3} and \ref{prop:lower-k-2} together complete the proof of Theorem \ref{thm:cliques}.

\medskip

\section{Counting cycles} \label{sec:cycles}

In this section, we study the game score when both the target and forbidden graphs are cycles. Specifically, we prove Theorem~\ref{thm:cycles} by determining the order of $g(n, C_t, C_{2k})$ for any $4 \le t$ with $t\neq 2k$. 

The upper bound follows directly from the static generalized extremal number. Since Constructor can never form a $C_{2k}$, her final graph must be $C_{2k}$-free. By the cycle-versus-cycle extremal theorem of Gerbner-Gy\H{o}ri-Methuku-Vizer \cite{GGMV20} and Gishboliner-Shapira \cite{GS20}, we have the  upper bound
\[
g(n, C_t, C_{2k}) \le \ex(n, C_t, C_{2k}) \le \begin{cases}
O_{t,k}\bigl(n^{t/2}\bigr), & \text{if } k=2,\\[2mm]
O_{t,k}\bigl(n^{\lfloor t/2\rfloor}\bigr), & \text{if } k\ge 3.
\end{cases}.
\]

To prove the matching lower bounds, we mirror the structure of the previous section. We split the analysis into two subsections. For longer even cycles where $k \ge 3$, we design a template graph. For the $4$-cycle where $k=2$, we extend our finite field algebraic construction.

\subsection{The lower bound for longer even cycles ($k \ge 3$)} \label{subsec:cycles-k3}

The goal of this subsection is to establish the lower bound for all $k \ge 3$. 

\begin{proposition} \label{prop:cycles-k-3}
Let $k \ge 3$ and $t \ge 4$, $t \neq 2k$, then the game score satisfies
\[
g(n,C_t,C_{2k})
\ge
\Omega_{t,k}\bigl(n^{\lfloor t/2\rfloor}\bigr).
\]
\end{proposition}

To achieve this bound, Constructor plays inside a template graph that is essentially a partial blow-up of the target cycle. We fix a small set of core vertices and build linear-size common neighbourhoods between consecutive core vertices. Choosing one vertex from each common neighbourhood will yield a distinct copy of the target cycle. 

We formalize this idea by defining two families of template graphs, depending on the parity of $t$.

\begin{definition}[Template graphs]
We define the template graphs $\mathcal{H}_t$ as follows. Let $r = \lfloor t/2 \rfloor$.
\begin{enumerate}
    \item If $t = 2r+1$, the graph has core vertices $\{u_0, u_1, \dots, u_r\}$ and disjoint auxiliary sets $A_1, \dots, A_r$. The edges are $u_{i-1}a$ and $u_i a$ for all $a \in A_i$ where $1 \le i \le r$, along with a single core edge $u_r u_0$.
    \item If $t = 2r$, the graph has core vertices $\{u_0, u_1, \dots, u_{r-1}\}$ and disjoint auxiliary sets $A_1, \dots, A_r$. The edges are $u_{i-1}a$ and $u_i a$ for all $a \in A_i$ where $1 \le i \le r-1$, along with the edges $u_{r-1}a$ and $u_0 a$ for all $a \in A_r$.
\end{enumerate}
\end{definition}

Before deploying this template, we must ensure it is safe from the forbidden cycle.

\begin{lemma} \label{lem:template-safe}
Any subgraph of a template graph from $\mathcal{H}_t$ contains only cycles of length $4$ or exactly $t$. Consequently, if $k \ge 3$ and $t \neq 2k$, every such template is  $C_{2k}$-free.
\end{lemma}
\begin{proof}
Any local cycle formed entirely within an auxiliary set $A_i$ and its two core neighbours $u_{i-1}$ and $u_i$ must be a $4$-cycle. Any cycle that is not strictly local must traverse through different auxiliary sets. Because of the tree-like structure connecting the sets, such a cycle is forced to go all the way around the core vertices, completing the full macroscopic cycle. By definition, this full cycle has length $t$. Since $2k \ge 6$, the forbidden length is not 4. Since we strictly assume $t \neq 2k$, the template cannot contain any $C_{2k}$.
\end{proof}

With the safety of the template guaranteed, we now provide the formal proof of Proposition \ref{prop:cycles-k-3}.

\begin{proof}[Proof of Proposition \ref{prop:cycles-k-3}]
We set the block size 
\[
s = \lfloor n/(300t) \rfloor.
\]
The entire strategy consumes at most $10rs<n/2$ vertices, so there is always an abundant supply of fresh vertices.

\begin{claim} \label{claim:build-block}
Suppose $a$ is a previously chosen core vertex. Constructor can securely choose a fresh core vertex $b$ and construct an auxiliary set $W$ of size $s$ such that she owns the edges $aw$ and $bw$ for all $w \in W$.
\end{claim}
\begin{poc}
Constructor first selects $4s$  fresh vertices. We call this set $A$. Whenever she selects a fresh vertex $x \in A$, she immediately claims the edge $ax$. This is  legal because $x$ is fresh. 

She then selects another  fresh vertex $b$. At this moment, every edge between $b$ and $A$ is unclaimed. Constructor repeatedly claims edges from $b$ to $A$ until she successfully acquires $s$ such edges. Since she only needs $s$ edges and Blocker can take at most one edge per round, Blocker can ruin at most $s$ candidate edges. Thus, $4s$ initial vertices provide sufficiently many vertices to guarantee $s$ successful connections. We define $W \subseteq A$ as the set of these successful connections.
\end{poc}

\noindent
\textbf{Odd target length.}
Assume
\(
        t=2r+1.
\)
Constructor starts by choosing a fresh vertex $u_0$.

For $i=1,2,\dots,r-1$, Constructor applies the ordinary block claim with
\(
        a=u_{i-1}.
\)
This gives a fresh new core vertex $u_i$ and a set $W_i$ of size $s$
such that Constructor owns all edges
\(
        u_{i-1}w,\ u_iw
\) where $w\in W_i$.

For the final block, Constructor first chooses $4s$ fresh vertices and claims all edges from $u_{r-1}$ to them; call this set $A_r$.  Then Constructor chooses a fresh vertex $u_r$ and immediately claims the edge
\(
        u_ru_0.
\)
  After Blocker's next move, Constructor starts claiming edges from $u_r$ to $A_r$ until she has obtained $s$ such edges.
 Let $W_r\subseteq A_r$ be the set of $s$ vertices to which Constructor successfully connects $u_r$.

At this point, for every $1\le i\le r$, Constructor has a set $W_i$ of size $s$, and for every $w\in W_i$ she owns the two edges from $w$ to the corresponding consecutive core vertices.  She also owns $u_ru_0$.  Therefore, for every tuple
\(
        (w_1,w_2,\dots,w_r)\in W_1\times W_2\times\cdots\times W_r,
\)
Constructor's graph contains the cycle
\(
        u_0w_1u_1w_2u_2\cdots u_{r-1}w_ru_ru_0.
\)
Hence Constructor has created at least
\(s^r=\Omega_t(n^r)
\)
copies of $C_{2r+1}=C_t$.

All edges claimed by Constructor during this strategy are contained in a graph of the form $\mathcal H_{2r+1}$ from the lemma, except that some vertices in the preparation sets $A_i\setminus W_i$ may have only one of their two possible incident template edges.  Equivalently, Constructor's graph is a subgraph of such a template.  By the lemma the template is $C_{2k}$-free, and every subgraph of a $C_{2k}$-free graph is $C_{2k}$-free.  Thus every move above is legal.

\medskip
\noindent
\textbf{Even target length.}
Assume
\(
        t=2r.
\)
Constructor starts by choosing a fresh vertex $u_0$.

If $r\ge 3$, then for $i=1,2,\dots,r-2$, Constructor applies the ordinary block claim with
\(
        a=u_{i-1}.
\)
This gives a fresh new core vertex $u_i$ and a set $W_i$ with $|W_i|=s$ such that Constructor owns all edges
\(
        u_{i-1}w,\ u_iw
\) where $w\in W_i$.
If $r=2$, this preliminary part is empty.

It remains to build the last two blocks, namely the block between $u_{r-2}$ and $u_{r-1}$ and the block between $u_{r-1}$ and $u_0$.  The two old endpoints $u_{r-2}$ and $u_0$ are already fixed, so Constructor prepares both stars before choosing the new core vertex $u_{r-1}$.

First, Constructor chooses $4s$ fresh vertices and claims all edges from $u_{r-2}$ to them.  Call this set $A_{r-1}$.  In the case $r=2$, this means $u_{r-2}=u_0$.

Second, Constructor chooses another $4s$ fresh vertices and claims all edges from $u_0$ to them.  Call this set $A_r$.  The sets $A_{r-1}$ and $A_r$ are disjoint because all vertices are chosen fresh.

Third, Constructor chooses a fresh vertex $u_{r-1}$.  Since $u_{r-1}$ is fresh, all edges from $u_{r-1}$ to
\(
        A_{r-1}\cup A_r
\)
are unclaimed at that moment.

Constructor now claims $s$ edges from $u_{r-1}$ to $A_{r-1}$.  The same counting argument as in the ordinary block shows that this is possible.  Let $W_{r-1}\subseteq A_{r-1}$ be the set of endpoints obtained; then
\(
        |W_{r-1}|=s.
\)

Constructor then claims $s$ edges from $u_{r-1}$ to $A_r$.  Let $W_r\subseteq A_r$ be the set of endpoints obtained; then
\(
        |W_r|=s.
\)

Now Constructor owns, for each $1\le i\le r-1$, all edges
\(
        u_{i-1}w,\ u_iw
\) where $w\in W_i$,
and she also owns all edges
\(
        u_{r-1}w,\ u_0w
\) where $W\in W_r$.
Therefore, for every tuple
\(
        (w_1,w_2,\dots,w_r)\in W_1\times W_2\times\cdots\times W_r,
\)
Constructor's graph contains the cycle
\(
        u_0w_1u_1w_2u_2\cdots u_{r-1}w_ru_0.
\)
  Hence Constructor has created at least \(
        s^r=\Omega_t(n^r)
\)
copies of $C_{2r}=C_t$.

As in the odd case, Constructor's graph throughout the strategy is a subgraph of a graph of the form $\mathcal H_{2r}$ from the lemma, possibly with some auxiliary vertices having only one of their two possible template edges.  The lemma says that $\mathcal H_{2r}$ is $C_{2k}$-free.  Hence Constructor never creates a $C_{2k}$, so every move is legal.

\medskip
Combining the odd and even strategies, Constructor can force at least
\(\Omega_t(n^r)=\Omega_t\bigl(n^{\lfloor t/2\rfloor}\bigr)
\)
copies of $C_t$ while keeping her graph $C_{2k}$-free. 
\end{proof}

\begin{remark}
Just like the vertex duplication strategy from the previous section, the partial blow-up template relies heavily on local $4$-cycles. Therefore, this structural approach is not applicable when the forbidden graph is $C_4$.
\end{remark}

\medskip

\subsection{The lower bound for the $4$-cycle ($k=2$)} \label{subsec:cycles-k2}

In this subsection, we establish the matching lower bound for the $4$-cycle by expanding our algebraic approach.

\begin{proposition} \label{prop:cycles-k-2}
For any integer $t \ge 5$, the game score for the $4$-cycle satisfies
\[
g(n, C_t, C_4) \ge \Omega_t\bigl(n^{ t/2 }\bigr).
\]
\end{proposition}

To prove this proposition, we reuse the finite field incidence graph $B$ defined in Section 2.2. Let $q$ be an odd prime power proportional to $\sqrt{n}$. Recall that $B$ is a bipartite graph on $q^2$ vertices that is strictly $C_4$-free. 

Before counting the cycles, Constructor will first build a dense skeleton within this geometry. We separate this initial phase into the following observation.

\begin{observation} \label{obs:dense-skeleton}
Constructor can securely build a $C_4$-free subgraph $G_B'$ on $q^2$ vertices with $\Omega(q^3)$ edges and a minimum degree proportional to $q$.
\end{observation}
\begin{proof}
Constructor first plays exclusively on the incidence graph $B$ until all its edges are claimed. Let $G_B$ be the graph of edges she successfully claims. Because she moves first, she claims at least half of the edges, giving 
\[
e(G_B) = \Omega(q^3).
\]
We then delete low-degree vertices from $G_B$ by repeatedly removing vertices of very low degree. This deletion process leaves a dense subgraph $G_B'$ that still has $\Omega(q^3)$ edges and maintains a minimum degree proportional to $q$. Since $G_B'$ is a subgraph of $B$, it is naturally $C_4$-free.
\end{proof}

With this dense skeleton established, we now complete the proof by analyzing the target cycles based on their parity.

\begin{proof}[Proof of Proposition \ref{prop:cycles-k-2}]
We divide the analysis into two cases.

\vspace{0.2cm}
\noindent\textbf{Case 1: Even target length ($t = 2r \ge 6$)}

For even target lengths, the bound follows immediately from classic supersaturation, requiring no additional algebraic edges. We consider the dense skeleton $G_B'$ obtained from Observation \ref{obs:dense-skeleton}. This graph has $q^2$ vertices and $\Omega(q^3)$ edges. 

By the even cycle supersaturation theorem of Simonovits, any $N$-vertex graph with $m \gg N^{1+1/r}$ edges must contain many copies of $C_{2r}$. We apply this theorem with $N = q^2$ and $m = \Omega(q^3)\gg N^{1+1/r}$ since $r\ge 3$. The number of copies of $C_{2r}$ is at least
\[
\Omega_r\left( \left(\frac{q^3}{q^2}\right)^{2r} \right) = \Omega_r(q^{2r}) = \Omega_r(n^r).
\]
Since $r = t/2$, this yields $\Omega_t(n^{t/2})$ copies of the target cycle. Because $G_B'$ is $C_4$-free, Constructor has already achieved her goal without claiming any further edges, and no further moves are needed.

\vspace{0.2cm}
\noindent\textbf{Case 2: Odd target length ($t = 2r+1 \ge 5$)}

Because the incidence graph $B$ is bipartite, the skeleton $G_B'$ contains no odd cycles. Therefore, Constructor must actively close paths to form odd cycles. We mirror the logic used to count triangles, but instead of closing paths of length two, we close paths of length $2r$.

\begin{claim}
The skeleton $G_B'$ contains a set $\mathcal{P}$ of $\Omega_r(q^{2r+2})$ ordered simple paths of length $2r$ that start and end at point vertices.
\end{claim}
\begin{poc}
Starting from any initial point vertex in $G_B'$, we can walk alternatingly between lines and points. Because the minimum degree is proportional to $q$, we have $\Omega(q)$ valid choices at each of the $2r$ steps. The number of self-intersecting paths is negligible. Since there are $\Omega(q^2)$ starting vertices, the total number of ordered simple paths of length $2r$ is 
\[
\Omega(q^2 \cdot q^{2r}) = \Omega_r(q^{2r+2}),
\]
as required.
\end{poc}

For each path $\Pi \in \mathcal{P}$ starting at a point $p_0 = (x_0, y_0)$ and ending at a point $p_r = (x_r, y_r)$, we define its closing parameter
\[
\theta(\Pi) = y_0 + y_r - x_0 x_r.
\]
By our earlier geometric definitions from Section 2.2, the edge $p_0 p_r$ belongs exactly to the auxiliary graph $F_{\theta(\Pi)}$. Adding this single edge to the path $\Pi$ completes a cycle of length $2r+1$.

\begin{claim}
There exists a specific $\theta \in \Fq$ such that the unclaimed edges of $F_\theta$ can close at least $\Omega_r(q^{2r+1})$ paths from $\mathcal{P}$.
\end{claim}
\begin{poc}
During the first phase of the game, Blocker claimed at most \(O(q^3)\)
edges between point vertices. We now bound how many paths can be ruined by
one fixed point--point edge. Fix two distinct points \(p_0,p_r\in P\).
Every ordered simple path of length \(2r\) in \(G_B'\) from \(p_0\) to
\(p_r\) has the form
\(
p_0,L_1,p_1,L_2,\dots,p_{r-1},L_r,p_r.
\)
Since \(G_B'\subseteq B\), every vertex has degree \(O(q)\). Hence the
prefix
\(
p_0,L_1,p_1,L_2,\dots,L_{r-1},p_{r-1}
\)
can be chosen in at most
\(
O_r(q^{2r-2})
\)
ways.

Once this prefix is fixed, the last line \(L_r\) must be a common
neighbour of the two points \(p_{r-1}\) and \(p_r\). Because the incidence
graph \(B\) is \(C_4\)-free, two distinct point vertices have at most one
common line. Since the path is simple, \(p_{r-1}\neq p_r\), and therefore
there is at most one possible choice for \(L_r\). Thus the number of
ordered simple paths from \(p_0\) to \(p_r\) is at most
\(
O_r(q^{2r-2}).
\)
The same bound holds if the orientation of the endpoint pair is reversed,
up to a factor of \(2\). Consequently, a single point--point edge can serve
as the closing edge for at most \(O_r(q^{2r-2})\) paths in \(\mathcal P\).
Therefore Blocker's point--point edges ruin at most
\[
O(q^3)\cdot O_r(q^{2r-2})=O_r(q^{2r+1})
\]
paths.
\end{poc}

Constructor fixes this specific $\theta$. She assigns each available edge in $F_\theta$ a weight equal to the number of paths it can close. She then claims these candidate edges by always selecting the heaviest available edge. As proven earlier in Lemma \ref{lem:polarity-C4-free}, the union graph $B \cup F_\theta$ is   $C_4$-free, so every move is legal. 

Blocker can take at most one candidate edge per round. Because Constructor prioritizes the heaviest edges, she secures at least half of the total weight. Therefore, she successfully completes
\[
\Omega_r(q^{2r+1}) = \Omega_r\bigl( n^{r + 1/2} \bigr)
\]
copies of $C_{2r+1}$. Since $t = 2r+1$, this exactly matches the required $\Omega_t(n^{ t/2 })$ bound, completing the proof. 
\end{proof}

\medskip

\section{Concluding remarks} \label{sec:concluding}

In this paper, we studied the Constructor--Blocker game when the forbidden graph is an even cycle. We determined the exact order of the game score when the target graph is a cycle of  length at least four. We also established general bounds when the target is a triangle, and these bounds determine the exact order for small forbidden even cycles.

For target cliques larger than a triangle, the game behaves differently. If the target clique is too large, Constructor cannot build even a single copy. Specifically, we prove in Appendix \ref{app:additional} that the game score is exactly zero for the target clique $K_{2k-1}$ when $k \ge 3$. 

However, there is a large gap in our understanding between the triangle target and the impossible large clique targets. The behaviour of the game score for general cliques remains unknown. We state this as our first open problem.

\begin{problem}
What is the correct order of the game score $g(n, K_r, C_{2k})$ for general values of  $k \ge 4$ and $3\le r\le 2k-2$?
\end{problem}

Beyond even cycles, exploring the game for other bipartite forbidden graphs is a natural next step. A prime candidate is the complete bipartite graph $K_{s,t}$. Currently, determining the exact order of the game score when forbidding $K_{s,t}$ appears very difficult. This difficulty mainly comes from the fact that the classical static extremal numbers for general complete bipartite graphs are still open. Without tight static lower bounds and flexible algebraic constructions, Constructor lacks a dense and safe environment to build the target graphs. As a result, the dynamic game naturally inherits the large gaps from traditional extremal graph theory. We leave the study of forbidding general bipartite graphs as a broad direction for future research.

\medskip

\section*{Acknowledgements}

LW was supported by the National Key R\&D Program of China under grant number 2024YFA1013900, the NSFC under grant number 12471327, the China Scholarship Council, and the Institute for Basic Science (IBS-R029-C4). ZY was supported by the Institute for Basic Science (IBS-R029-C4).


\medskip

\appendix

\section{Additional results} \label{app:additional}

In this appendix, we prove that the game score is exactly zero when the target is a large clique and the forbidden graph is an even cycle. Specifically, we prove the following theorem.

\begin{theorem} \label{thm:zero-score}
For every fixed $k\ge 3$,
\[
g(n,K_{2k-1},C_{2k})=0.
\]
\end{theorem}

To prove this theorem, we first establish a structural lemma regarding graphs that are missing exactly one edge from a complete graph. Let $K_r^-$ denote a graph on $r$ vertices isomorphic to a complete graph missing exactly one edge.

\begin{lemma} \label{lem:shared-edge}
Let $r \ge 5$. Suppose $H_1$ and $H_2$ are two copies of $K_r^-$ on distinct vertex sets $V_1$ and $V_2$. If $H_1$ and $H_2$ share at least one edge $f=uv$, then their union $H_1 \cup H_2$ contains a cycle of length $r+1$.
\end{lemma}
\begin{proof}
Because $V_1 \neq V_2$, there is at least one vertex $y \in V_2 \setminus V_1$. We first show that we can always find such a $y$ that is adjacent to at least two distinct vertices in $V_1 \cap V_2$. 

Since $H_1$ and $H_2$ share the edge $f=uv$, we know $u, v \in V_1 \cap V_2$.
If $V_1 \cap V_2 = \{u, v\}$, then $|V_2 \setminus V_1| = r - 2 \ge 3$. The graph $H_2$ is missing at most one edge in total, which means this missing edge is incident to at most two vertices in the entire graph. Since $|V_2 \setminus V_1| \ge 3$, by the pigeonhole principle, there must be at least one vertex $y \in V_2 \setminus V_1$ that is not incident to the missing edge of $H_2$. Therefore, $y$ is adjacent to both $u$ and $v$.
If $|V_1 \cap V_2| \ge 3$, we pick any $y \in V_2 \setminus V_1$. Since $y$ misses at most one edge in $H_2$, it is adjacent to all but at most one vertex in $V_1\cap V_2$. Because $|V_1 \cap V_2| \ge 3$, $y$ is adjacent to at least two distinct vertices in $V_1 \cap V_2$.

In either case, we have a vertex $y \in V_2 \setminus V_1$ and two distinct vertices $a, b \in V_1$ such that $ya$ and $yb$ are edges in $H_2$.

Now consider the graph $H_1$. For any $r \ge 5$, the graph $K_r^-$ is Hamiltonian connected. To see this, observe that the sum of the degrees of the two non-adjacent vertices in $K_r^-$ is $(r-2) + (r-2) = 2r-4$. Since $r \ge 5$, we have $2r-4 \ge r+1$. By Ore's theorem for Hamiltonian-connected graphs, this guarantees a Hamiltonian path between any pair of distinct vertices. Therefore, there exists a path in $H_1$ from $a$ to $b$ that visits every vertex in $V_1$ exactly once. This path has length $r-1$. By adding the edges $ya$ and $yb$ to this path, we form a cycle of length exactly $r+1$.
\end{proof}

With this structural tool established, we now provide the proof of the main theorem of this appendix.

\begin{proof}[Proof of Theorem \ref{thm:zero-score}]
Let $r = 2k-1$. Since $k \ge 3$, we have $r \ge 5$. We describe a strategy for Blocker that prevents Constructor from ever completing a copy of $K_r$.

We call an unclaimed edge $e$ an open threat if adding $e$ to Constructor's current graph would create a copy of $K_r$. This implies Constructor's graph already contains a copy of $K_r^-$, and $e$ is the unique missing edge. Blocker maintains the invariant that, immediately after his turn, there are no open threats on the board. Since the board is initially empty, the invariant trivially holds at the start.

Suppose the invariant holds, and Constructor has just claimed an edge $f$. First, observe that Constructor cannot complete a $K_r$ with this single move. By the invariant, there were no open threats on the board prior to her turn, meaning no $K_r^-$ was waiting for its final edge.

If this move creates any new open threats, these threats must rely on $f$. Thus, $f$ must belong to any newly formed $K_r^-$.

Assume for contradiction that Constructor's move creates two distinct open threats, $e_1$ and $e_2$. This means her graph now contains two copies of $K_r^-$, say $H_1$ missing $e_1$ and $H_2$ missing $e_2$, and both contain $f$. 

If $V(H_1) = V(H_2)$, the set of claimed edges on this vertex set must be exactly the edges of $K_r$ excluding $e_1$ and $e_2$. However, for $e_1$ to be an open threat, adding it must complete a $K_r$, which means all other edges must already be claimed. This forces $e_2$ to be a claimed edge, contradicting the definition that $e_2$ is an unclaimed open threat. Thus, $V(H_1) \neq V(H_2)$.

Since $H_1$ and $H_2$ are distinct copies of $K_r^-$ on different vertex sets sharing the edge $f$, Lemma \ref{lem:shared-edge} states that their union contains a cycle of length $r+1$. Since $r = 2k-1$, this cycle has length $2k$. This means Constructor's graph now contains a $C_{2k}$. 

However, the rules of the game strictly forbid Constructor from claiming any edge that creates a $C_{2k}$. Therefore, a legal move by Constructor can never create two open threats simultaneously. Since every legal move creates at most one open threat, Blocker can simply claim that unique threat on his turn. This maintains the invariant, ensuring Constructor can never complete a $K_{2k-1}$.
\end{proof}


\medskip

\begin{thebibliography}{GGMV20}

\bibitem[BCE25]{BCE25}
J.~Balogh, C.~Chen, and S.~English.
\newblock On the constructor-blocker game.
\newblock {\em Journal of Graph Theory}, 108, 2025.

\bibitem[Ben66]{B}
C.~T. Benson.
\newblock Minimal regular graphs of girths eight and twelve.
\newblock {\em Canadian Journal of Mathematics}, 18:1091--1094, 1966.

\bibitem[BMPP25]{BMPP25}
Chlo{\'e} Boisson, Yannick Mogge, Aline Parreau, and Th{\'e}o Pierron.
\newblock Creating triangles in {C}onstructor-{B}locker games.
\newblock {\em arXiv preprint arXiv:2510.05811}, 2025.

\bibitem[BS74]{BS}
J.~A. Bondy and M.~Simonovits.
\newblock Cycles of even length in graphs.
\newblock {\em Journal of Combinatorial Theory, Series B}, 16(2):97--105, 1974.

\bibitem[Con21]{Conlon21}
D.~Conlon.
\newblock Extremal numbers of cycles revisited.
\newblock {\em Amer. Math. Monthly}, 128(5):464--466, 2021.

\bibitem[ERS66]{ERS}
P.~Erd{\H{o}}s, A.~R{\'e}nyi, and V.~T. S{\'o}s.
\newblock On a problem of graph theory.
\newblock {\em Studia Scientiarum Mathematicarum Hungarica}, 1:215--235, 1966.

\bibitem[GGMV20]{GGMV20}
D.~Gerbner, E.~Gy{\H{o}}ri, A.~Methuku, and M.~Vizer.
\newblock Generalized {T}ur{\'a}n problems for even cycles.
\newblock {\em Journal of Combinatorial Theory, Series B}, 145:169--213, 2020.

\bibitem[GS20]{GS20}
L.~Gishboliner and A.~Shapira.
\newblock A generalized tur{\'a}n problem and its applications.
\newblock {\em International Mathematics Research Notices}, 2020(11):3417--3452, 2020.

\bibitem[KNS64]{KNS}
G.~O.~H. Katona, T.~Nemetz, and M.~Simonovits.
\newblock On a graph-problem of tur{\'a}n.
\newblock {\em Matematikai Lapok}, 15:228--238, 1964.

\bibitem[PSV24]{PSV24}
B.~Patk{\'o}s, M.~Stojakovi{\'c}, and M.~Vizer.
\newblock The constructor-blocker game.
\newblock {\em Applicable Analysis and Discrete Mathematics}, 18:1--15, 2024.

\bibitem[Wen91]{W}
R.~Wenger.
\newblock Extremal graphs with no \(c^4\)'s, \(c^6\)'s, or \(c^{10}\)'s.
\newblock {\em Journal of Combinatorial Theory, Series B}, 52(1):113--116, 1991.

\end{thebibliography}
\end{document}